\documentclass[12pt]{article}
\usepackage{amssymb}
\usepackage{stmaryrd}
\usepackage[english]{babel}
\usepackage{comment}
\usepackage{amsmath}

\usepackage{color}
\usepackage{graphicx} 
\usepackage{environ}

\newtheorem{theorem}{Theorem}[section]

\newtheorem{e-proposition}[theorem]{Proposition}

\newtheorem{remark}{\it Remark\/}
\newtheorem{example}{\it Example\/}
\NewEnviron{myequation}{%
\begin{equation}
\scalebox{1.2}{$\BODY$}
\end{equation}
}

\def\Sym{\mathrm{Sym}}
\def\NCSF{\mathbf{NCSF}}
\def\WSym{\mathbf{WSym}}
\def\CWSym{\mathbf{CWSym}}
\def\std{\mathrm{std}}

\def\<{\langle}
\def\>{\rangle}

\def\C{{\mathbb C}}

\def\N{{\mathbb N}}

\def\A{{\sf A}}

\def\A{{\mathbb A}}

\def\X{{\mathbb X}}

\def\sp{\mathrm{span}}

\def\binom#1#2{\left(#1\atop#2\right)}

\include {mak}
\begin{document}
%\maketitle
\title{$r$-Bell polynomials in combinatorial Hopf algebras}
\author{Ali Chouria\footnote{ali.chouria1@univ-rouen.fr, ali.chouria@yahoo.fr. Laboratoire LITIS - EA 4108, Université de Rouen, Avenue de l'Université - BP 8,
76801 Saint-\'Etienne-du-Rouvray Cedex }  and Jean-Gabriel Luque\footnote{jean-gabriel.luque@univ-rouen.fr. Laboratoire LITIS - EA 4108, Université de Rouen, Avenue de l'Université - BP 8, 76801 Saint-Étienne-du-Rouvray Cedex }}
\maketitle

\begin{abstract}
\selectlanguage{english} 
We introduce partial $r$-Bell polynomials in three combinatorial Hopf algebras. We prove a factorization formula for the generating functions which is a consequence of the Zassenhauss formula.
\vskip 0.5\baselineskip

% Text of abstract in French
\noindent{\bf R\'esum\'e} \vskip 0.5\baselineskip \noindent
{\bf Les $r-$polyn\^omes de Bell dans des alg\`ebres de Hopf combinatoires}\\
Nous d\'efinissons des polyn\^{o}mes $r$-Bell partiels dans trois alg\`ebres de Hopf combinatoires. Nous prouvons une formule de factorisation pour les fonctions g\'en\'eratrices qui est une cons\'equence de la formule de Zassenhauss. 
\end{abstract}

\noindent{\footnotesize {\bf Keywords:} {$r-$Bell polynomials, Munthe-Kaas polynomials, Set partition, Bicolored partitions, Hopf algebras, Bisymmetric functions, Noncommutative bisymmetric functions, Zassenhauss formula, Generating functions.}

% main text
\section{Introduction \label{sec1}}
\emph{Partial multivariate Bell polynomials} have been defined by E.T. Bell \cite{bell1934exponential} in 1934.
Their applications in Combinatorics, Analysis, Algebra, Probabilities \emph{etc.} are numerous (see \emph{e.g} \cite{ThesisMih}).
They are usually defined on an infinite set of commuting variables $\{a_1,a_2,\ldots\}$ by the following generating function:
\begin{myequation}\label{Bell2}
\sum_{n \geqslant 0} B_{n,k}(a_{1},\dots,a_{p},\dots) \frac{x^{n}}{n!}t^{k} = \exp\left\{\sum_{m \geqslant 1} a_{m} \frac{x^{m}}{m!}t\right\}.
\end{myequation}

The partial Bell polynomials are related to several combinatorial sequences. 
Let ${n \brace k}$ denotes the Stirling number of second kind which counts the number of ways to partition a set of $n$ objects into $k$ nonempty subsets and let ${n \brack k}$ denotes the Stirling number of first kind  which counts the number of permutations according to their number of cycles.
We have,
$B_{n,k}(1,1, \dots)=  {n \brace k}$ and $B_{n,k}(0!,1!,2!, \dots)= {n \brack k}$.

The connection between the Bell polynomials and the combinatorial Hopf algebras has been investigated by one of the authors in \cite{Chouria}.

In the aim to generalize these polynomials, \emph{Mihoubi et al.}  \cite{mihoubi2013partial} defined partial $r-$Bell polynomials  by setting
\begin{myequation}\label{th3h}
B^{r}_{n+r,k+r}(a_1, a_2, \cdots; b_1, b_2, \cdots) =  
\sum_{\substack{ n'+n'' = l \\ k'+k'' = k}} \sum_{\substack{ \lambda'_1 + \cdots + \lambda'_{r} = n' \\ \lambda''_1 + \cdots + \lambda''_{n-r} = n''}} \alpha_{\lambda', \lambda''}^{r} a_{\lambda'_1} \cdots a_{\lambda'_{r}} b_{\lambda''_1 \cdots b_{\lambda''_{n-r}}},
\end{myequation}
where $\alpha_{\lambda', \lambda''}^{r}$ is the number of sets partition $\pi =  \{ \pi'_1, \pi'_2, \cdots,  \pi'_r, \pi''_1, \pi''_2, \cdots,  \pi''_{k-r}\}$ of $\{1,2, \cdots, n\}$ such that 
$\# \pi'_1 = \lambda'_1, \cdots, \# \pi'_{r} = \lambda'_{r}, \# \pi''_1 = \lambda''_1, \cdots, \# \pi''_{n-r} = \lambda''_{n-r}$ and $1 \in  \pi'_1, 2 \in  \pi'_2, \cdots, r \in  \pi'_r$ and $\# \pi_i$ denotes the cardinalty of $\pi_i$. 
Comparing our notation to those of \cite{mihoubi2013partial},  the role of the variables $a_{i}$ and $b_{i}$ have been switched. The generating function of the $r-$Bell polynomials is known to be
\begin{myequation}\label{th3Mih} \begin{array}{l}
\sum_{n \geqslant k} B^{r}_{n+r,k+r}(a_1, a_2, \cdots; b_1, b_2, \cdots) \frac{x^n}{n!} \frac{y^r}{r!} t^k \\ %&= 
%\sum \frac{t^k}{k!} \frac{y^r}{r!} \left(\sum_{j \geqslant 0} a_{j+1} \frac{x^j}{j!}\right)^r\left( \sum_{j \geqslant 1} b_j \frac{x^j}{j!}\right)^k  \\
= \exp\left(\sum_{j \geqslant 0} a_{j+1} \frac{x^j }{j!}y\right)\exp\left( \sum_{j \geqslant 1} b_j \frac{x^j t}{j!}\right),
\end{array} \end{myequation}
where $(a_n; n \geqslant 1)$ and $(b_n; n \geqslant 1)$ are two sequences of nonnegative integers. 

The aim of our paper is to show that we can define three versions of the $r$-Bell polynomials in three combinatorial Hopf algebras in the same way. The first algebra is $\Sym^{(2)}$, the algebra of bisymmetric functions (or symmetric functions of level $2$). The $r$-Bell polynomials as defined by \emph{Mihoubi} belong to this algebra. 
 The second algebra is $\NCSF^{(2)}$, the algebra of noncommutative bisymmetric functions. In this algebra, we define non commutative analogues of  $r$-Bell polynomials which generalize the Munthe-Kaas polynomials. The third algebra is the $\WSym^{(2)} := \CWSym(2,2,\cdots)$, the algebra of $2$-colored word symmetric functions. In this algebra, we define word analogue of $r$-Bell polynomials. The common of the three constructions is that they are based on the same algorithm which generates $2$-colored set partitions without redundance. Our main result is a factorization formula for the generating function which holds in the three algebras and which is a consequence of the Zassenhauss formula.

\section{Bi-colorations of partitions, compositions and set partitions}\label{Partie_1}
A bicolored partition $\lambda$ of $n$ is a multiset $\{(\lambda_{1},j_{1}),\dots,(\lambda_{k},j_{k})\}$ such that $\lambda_{1}+\cdots+\lambda_{k}=n$ and $j_{1},\dots,j_{k}\in \{1,2\}$. We set $\lambda\vdash n$, $\omega(\lambda)=n$ and $\ell(\lambda)=k$. A bicolored composition $I$ of $n$ is a list $I=[(i_{1},j_{1}),\dots,(i_{k},j_{k})]$ with $i_{1}+\cdots + i_{k}=n$ and  $j_{1},\dots,j_{k}\in \{1,2\}$. We set $I\vDash n$, $\omega(I)=n$ and $\ell(I)=k$.
A bicolored set partition is a set $\pi=\{(\pi_1,j_{1}),\dots,(\pi_k,j_{k})\}$ such that $\{\pi_{1},\dots,\pi_{k}\}$ is a set partition of $n$ and  $j_{1},\dots,j_{k}\in \{1,2\}$. We set $\pi\Vdash n$, 
$\omega(\pi)=n$ and $\ell(\pi)=k$.

We define 
\begin{myequation}\label{Partition1} \begin{array}{l}
S_{n+r,k+r}^{r} = \{   \pi=\{ (\pi_1,1), \cdots,  (\pi_{r},1),  (\pi_{r+1},2),  \cdots,  (\pi_{k+r},2)\} : \\ \pi\Vdash (n+r), 1 \in \pi_1, 2 \in \pi_2, \cdots, r \in \pi_r  \}.
\end{array} \end{myequation}

We have
\begin{myequation} 
S_{r,r}^{r} = \left\lbrace  \{(\{1\},1), (\{2\},1), \cdots, (\{r\},1)\} \right \rbrace 
\end{myequation} 
and
\begin{myequation}\label{resS} \begin{array}{l}
 S_{n+1+r,k+r}^{r} = \left\{ \pi \cup  \{(n+1,2)\} : \pi \in S_{n+r,r+k-1}^{r} \right\}\cup
 \\
  \{ \pi\setminus \{(\pi_{\ell},j_{\ell})\}\cup\{(\pi_{\ell}\cup\{n+1\},j_{\ell})\} :  \\
    \pi=\{(\pi_1,j_{1}), \cdots,  (\pi_{r+k},j_{r+k}), 1\leq\ell\leq r+k\} 
    \in S_{n+r,k}^{r+k} \}. %, \forall 1 \leqslant i \leqslant n.
 \end{array}\end{myequation}
The underlying recursive algorithm generates one and only one times each element of $ S_{n+1+r,k+r}^{r}$.
We consider also two applications: $c(\pi)=[(\#\pi_{1},j_{1}),\dots,(\#\pi_{k},j_{k})]$ if $\pi=\{(\pi_{1},j_{1}),\dots,(\pi_{k},j_{k)}\}$
 with $\min\{\pi_{1}\}<\cdots<\min\{\pi_{k}\}$ and 
$\lambda(\pi)=\{(\#\pi_{1},j_{1}),\dots,(\#\pi_{k},j_{k})\}$. 
We define 
\begin{myequation}
f^{r}_{n+r,k+r}(I)= \# \{\pi\in S_{n+1+r,k+r}^{r}:c(\pi)=I\} 
\end{myequation}
 and 
\begin{myequation} 
 g^{r}_{n+r,k+r}(\lambda)= \#\{\pi\in S_{n+1+r,k+r}^{r}:\lambda(\pi)=\lambda\}.
\end{myequation}

\section{Three combinatorial Hopf algebras}\label{Partie_3}
\subsection{Algebras of symmetric functions of level 2}
In this section, we define three combinatorial Hopf algebras indexed by bicolored objects.
 The model of construction is the algebra $\Sym^{(l)}$ which is the representation ring of a wreath product $(\Gamma\wr \mathfrak S_{n})_{n\geq 0}$, $\Gamma$ being a group with $l$ conjugacy classes \cite{macdonald1998symmetric}. Let us recall briefly its definition for $l=2$.
The combinatorial Hopf algebra $\Sym^{{(2)}}$ \cite{macdonald1998symmetric} is naturally realized as symmetric functions in $2$ independent sets of variables $\Sym^{(2)}:= \Sym(\X^{1};\X^{2})$. It is the free commutative algebra generated by two sequences of formal symbols $p_{1}(\X^{(1)}), p_{2}(\X^{(1)}),\dots$ and $p_{1}(\X^{(2)}), p_{2}(\X^{(2)}),\dots$, named power sums,  which are primitive for its coproduct. The set of the polynomials $p^{\lambda}:=p_{\lambda_{1}}(\X^{(1)})\cdots p_{\lambda_{k}}(\X^{(k)})$ where $\lambda=\{(\lambda_{1},i_{1}),\dots,(\lambda_{2},i_{2})\}$  is a bicolored partition. 

The Hopf algebra $\NCSF$ of formal noncommutative symmetric functions \cite{ncsf} is the free associative algebra $\mathbb C\left\langle \Psi_1, \Psi_2, \cdots \right\rangle$ generated by an infinite sequence of primitive formal variables $(\Psi_i)_{i \geqslant 1}$. Its level $l$ analogous is described in  \cite{NT} as the free algebra generated by level $l$ complete homogeneous functions $S_{\mathbf n}$. 
Here we set $l=2$ and we use another basis. We recall that the level $2$ complete homogeneous function $S_{\mathbf n}$, for $\mathbf n\in\N^{2}$, is defined as a free quasi-symmetric function of level $2$ as 
$S_{\mathbf n}=\sum_{|u_{i}|=n_{i}}\mathbf G_{1\dots n,u}$ where $\mathbf G_{\sigma,u}$ denotes the dual free $l$-quasi-ribbon labeled by the colored permutation $(\sigma,u)$ \cite{NT}. 
Notice that $\mathbf G_{\sigma,u}$ is realized as a polynomial in $\C\langle \mathbb A^{(1)}\cup \mathbb A^{(2)}\rangle$, where $\mathbb A^{(i)}$ denotes two disjoint copies of the same alphabet $\mathbb A$
as
\begin{myequation}
\displaystyle\mathbf{ G}_{\sigma,u}=\sum_{w\in (\mathbb A^{(1)}\cup\mathbb A^{(1)})^{n} \atop std(w)=\sigma,w_{i}\in \mathbb A^{(i)}} w
\end{myequation}

 where $\std$ is the usual standardization applied after identifying the two alphabets $\mathbb A^{(1)}$ and $\mathbb A^{(2)}$.
Alternatively, for dimensional reasons, $\NCSF^{(2)}$ is the minimal sub (free) algebra of $\mathbb C\langle\mathbb A^{(1)}\cup \mathbb A^{(2)}\rangle$ containing both $\NCSF(\mathbb A^{(1)})$ and $\NCSF(\mathbb A^{(2)})$ as subalgebras. Hence, it is freely generated by the (primitive) power sums $\Psi_{i}(\mathbb A^{(j)})$.
 If $I=[(i_{1},j_{1}),\dots,(i_{k},j_{k})]$ denotes a bi-colored composition, then the set of the polynomials  $\Psi^{I}=\Psi_{i_{1}}(\mathbb A^{(j_{1})})\cdots{} \Psi_{i_{k}}(\mathbb A^{(j_{k})})$ is a basis of the space $\NCSF^{(2)}$.

The last algebra, $\WSym^{(2)}$, is a level $2$ analogue of the algebra of word symmetric functions introduced by \emph{M. C. Wolf} \cite{wolf1936symmetric} in 1936. We construct it as a special case  of the Hopf algebras $\CWSym(a)$ of colored set partitions introduced in \cite{BellHopf} for $a=(2,2,\dots,2,\dots)$.
As a space $\CWSym(a)$ is generated by the set $\Phi^\pi$ where $\pi$ denotes a bicolored set partition. 
Its product is defined by
\begin{myequation}\label{Cwsymproduct}
\Phi^\pi \Phi^{\pi'}=\Phi^{\pi\hat\cup\pi'},
\end{myequation}
where $\hat\cup$ denotes the shifted union obtained by shifting first the elements of $\pi'$ by the weight of $\pi$ and hence compute the union, and its coproduct is
\begin{myequation}\label{Cwsymcoproduct}
\Delta(\Phi^\pi) = \sum\limits_{\hat\pi_1\cup\hat\pi_2=\pi\atop \hat\pi_1\cap\hat\pi_2=\emptyset}\Phi^{\std(\hat\pi_1)}\otimes\Phi^{\std(\hat\pi_2)},
\end{myequation}
where the \emph{standardized} $\std(\pi)$ of $\pi$ is defined as the unique colored set partition obtained by replacing the $i$th smallest integer in the $\pi_j$ by $i$. 

The algebra $\Sym^{(2)}$ (resp. $\NCSF^{(2)}$, $\WSym^{(2)}$) is naturally bigradued $\Sym^{(2)}=\bigoplus_{n,k}\Sym^{(2)}_{n,k}$ (resp. $\NCSF^{(2)}=\bigoplus_{n,k}\NCSF^{(2)}_{n,k}$, $\WSym^{(2)}=\bigoplus_{n,k}\WSym^{(2)}_{n,k}$) where $\Sym^{(2)}_{n,k}=
\sp\{p^{\lambda}:\ell(\lambda)=k, \omega(\lambda)=n\}$ (resp. 
$\NCSF^{(2)}_{n,k}=\sp\{\Psi^{I}:\ell(I)=k, \omega(I)=n\}$,\\ $\WSym^{(2)}_{n,k}=\sp\{\Phi^{\pi}:\ell(\pi)=k,{} \omega(\pi)=n\}$).
We  denote by $\mathbb R$ the subalgebra of $\Sym^{(2)}$ (resp. $\NCSF ^{(2)}$, $\WSym^{(2)}$) spanned by the polynomials $p^{\{(\lambda_{1},2),\dots,(\lambda_{k},2)\}}$ (resp. $\Psi^{[(i_{1},2),\dots,(i_{k},2)]}$, $\Phi^{\{(\pi_{1},2),\dots,(\pi_{k},2)\}}$) which is isomorphic to $\Sym$ (resp. $\NCSF$, $\WSym$). Notice also that $\mathbb R=\bigoplus_{n,k}\mathbb R_{n,k}$ is naturally bigraded.

In the rest of the paper, when there is no ambiguity, we  use $a_i$ to refer to $p_i(\X^{(1)})$, $\Psi_i(\A^{(1)})$ or $\Phi^{\{(\{1,\dots,n\},1)\}}$ and $b_i$ to refer to $p_i(\X^{(2)})$, $\Psi_i(\A^{(2)})$ or $\Phi^{\{(\{1, \dots,n\},2)\}}$.
 Notice that with this notation all the $a_{i}$ and the $b_{i}$ are primitive elements.
  We define the natural linear map $\Xi:\WSym^{(2)}\rightarrow \NCSF^{(2)}$ and $\xi:  \WSym^{(2)}\rightarrow \Sym^{(2)}$ by $\Xi(\Phi^{\pi})=\Psi^{c(\pi)}$ and $\xi(\Phi^{\pi})=p^{\lambda(\pi)}$. Notice that these maps are morphisms of Hopf algebras.

\subsection{$r-$Bell polynomials and (commutative/ noncommutative/ word) symmetric functions}
In $\Sym^{(2)}$ and $\NCSF^{(2)}$, we define the operator $\partial$ as the unique derivation acting on its left and satisfying $a_i \partial= a_{i+1}$ and $b_i \partial=b_{i+1}$. In $\WSym^{(2)}$, we define $\partial$ as the operator acting linearly on the left by $1\partial=0$ and
\begin{myequation}
\Phi^{\{[\pi_1, i_1],\dots,[\pi_k, i_k]\}}\partial=\sum_{j=1}^k\Phi^{(\{[\pi_1, i_1],\dots,[\pi_k,i_k]\}\setminus [\pi_j,i_j])\cup\{[\pi_j\cup\{n+1\}, i_j]\}}.
\end{myequation}
In the three algebras, we define  $r-$Bell polynomials in a similar way to Ebrahimi-Fard \emph{et al.}
 defined Munthe-Kaas polynomials, that is by the use of the operator $\partial$. More precisely, the polynomial $B_{n+r,k+r}^r$
 is the coefficient of
 $t^{k}$ in $a_1^r (t b_1 + \partial)^{n}.$
In $\WSym^{(2)}$, from (\ref{resS}), we have 
\begin{myequation}
\displaystyle B^{r}_{n+r,k+r}=\sum_{\pi\in S_{n+r,k+r}^{r}}\Phi^{\pi}.
\end{myequation}
Hence, using the map $\Xi$ and $\xi$ we obtain 
\begin{myequation}
\displaystyle B^{r}_{n+r,k+r}=\sum_{\pi\in S_{n+r,k+r}^{r}}p^{\lambda(\pi)}=
\sum_{\lambda}g^{r}_{n+r,k+r}(\lambda)p^{\lambda}
\end{myequation}
in  $\Sym^{(2)}$ and
\begin{myequation}
\displaystyle B^{r}_{n+r,k+r}=\sum_{\pi\in S_{n+r,k+r}^{r}}\Psi^{\lambda(\pi)}=
\sum_{I}f^{r}_{n+r,k+r}(I)\Psi^{I}
\end{myequation}
 in $\NCSF^{(2)}$.
  Notice that in $\Sym^{{(2)}}$, $B^{r}_{n+r,k+r}$ is nothing but the classical $r$-Bell polynomial and in $\NCSF^{(2)}$, it is a $r$-version of Munthe-Kaas polynomial \cite{MK}.

\begin{example}\rm
In $\WSym^{(2)}$, we have 
\begin{align*}
B^2_{4,3} &=\Phi^{\{(\{1,3\},1),(\{2\},1),(\{4\},2)\}} + \Phi^{\{(\{1,4\},1),(\{2\},1),(\{3\},2)\}}
 + \Phi^{\{(\{1\},1),(\{2,3\},1),(\{4\},2)\}} 
 \\& +  \Phi^{\{(\{1\},1),(\{2,4\},1),(\{3\},2)\}} + \Phi^{\{(\{1\},1),(\{2\},1),(\{3,4\},2)\}}.
\end{align*}
In $\NCSF^{(2)}$, we have 
\begin{align*}
B^2_{4,3}=2\Psi^{[(2,1),(1,1),(1,2)]}+2\Psi^{[(1,1),(2,1),(1,2)]}+\Psi^{[(1,1),(1,1),(2,2)]}=2a_{2}a_{1}b_{2}+
2a_{1}a_{2}b_{1}+a_{1}a_{1}b_{2}.
\end{align*}
In $\Sym^{(2)}$,
$B^2_{4,3}=4p^{\{(2,1),(1,1),(1,2)\}}+p^{\{(1,1),(1,1),(2,2)\}}=4a_{2}a_{1}b_{2}+a_{1}a_{1}b_{2}$.
  \end{example}
  We consider also the polynomials $\tilde B^{r}_{n+k+r,k+r}=a_{1}^{r}b_{1}^{k}\partial^{n}.$ Notice that in $\WSym^{(2)}$, 
  we have
  \begin{myequation}
  \displaystyle \tilde B^{r}_{n+k+r,k+r}=\sum_{\{(\pi_1,1),\dots,(\pi_{k+r},1)\atop 1\in\pi_{1},\dots,r\in\pi_{r}\}\in S^{k+r}_{n+k+r,k+r}}\Phi^{\{(\pi_1,1),\dots,(\pi_{r},1),(\pi_{r+1},2),\dots,(\pi_{r+k},2)\}}.
  \end{myequation}

\section{Generating functions}

We consider the following generating functions:
\begin{myequation}\label{Stxy}
S(t,x,y) = \sum_{n,r,k} B_{n+r,k+r}^{r} \frac{x^n}{n!} \frac{y^r}{r!} t^k 
=  \exp{(a_1 y)} \exp{(x(t b_1 + \partial))},
\end{myequation}

\begin{myequation}\label{Scirc}
S^\circ(t,x) = S(t,x,0) = \sum_{n,k} B_{n,k} \frac{x^n}{n!} t^k 
=  1. \exp{(x(t b_1+ \partial))},
\end{myequation}

\begin{myequation}
S^\bullet (t,x,y) = \sum_{n,r,k} \tilde B_{n+k+r,k+r}^{r} \frac{x^n}{n!} \frac{y^r}{r!} \frac{t^k}{k!}  
=  \exp{(a_1 y)} \exp{(t b_1 )} \exp{(x\partial)},
\end{myequation}
and
\begin{myequation}
S^*(x,y) =  \sum_{n,r} B_{n+r,r}^{r} \frac{x^n}{n!} \frac{y^r}{r!}  
= \exp{(yb_1)} \exp{(x\partial)}.
\end{myequation}

\begin{theorem}\label{th1}
The generating functions $S(t,x,y)$ and $S^\circ(t,x)$ satisfy the following factorization
\begin{myequation}\label{Zass}
S(t,x,y) = S^\bullet(xt,x,y) Z(x,t)\mbox{ and }S^\circ(t,x)=S^*(x,xt) Z(x,t)
\end{myequation}
where  $Z(x,t) = \prod_{n \geqslant 2} \exp{(x^n \sum_k t^k C_{n,k})}$, $C_{n,k} = \frac{(-1)^{n+1}}{n}\frac{1}{k!(n-k-1)!} ad_{\partial}^{n-k-1} ad_{b_1}^{k} \partial$ and $ad_{x}$ is the derivation
 $ad_{x} P = [x,P] = x P - P x$.{}
 In $\Sym^{(2)}$ and $\NCSF^{(2)}$ the operator $C_{n,k}$ is the multiplication by a primitive polynomials belonging to the subalgebra $\mathbb R_{n,k}$.
 
\end{theorem}

\begin{proof}
Equalities (\ref{Zass}) are obtained from (\ref{Stxy}) and (\ref{Scirc}) by using \emph{Zassenhaus} formula \cite{magnus1954}.
 In $\Sym^{(2)}$ and $\NCSF^{(2)}$, since $\partial$ is a derivation, $ad_{\partial}^{i} ad_{b_1}^{j} \partial$ is primitive. 
 Furthermore,  remarking that $[b_i, \partial] = b_{i+1}$, we prove that $ad_{\partial}^{i} ad_{b_1}^{j} \partial\in\mathbb R_{n,k}$.
\end{proof}

\begin{example}
	\rm{}
	In $\NCSF^{(2)}$, consider the coefficient of $\frac {x^{3}}{3!}\frac {y^{2}}{2!}t$ in the left equality of (\ref{Zass}). In the left hand side we find $B^{2}_{5,3}=3a_{2}a_{1}b_{1}^{2}+3a_{1}a_{2}b_{1}^{2}+2a_{1}^{2}b_{2}b_{1}+a_{1}^{2}b_{1}b_{2}$. 
	The same coefficient in the right hand sides is
	${}
	3\tilde B^{2}_{5,4}-3\tilde B^{2}_{3,3}C_{2,1}+3!\tilde B^{2}_{2,2}C_{3,2}.
	$\\
	 Since $\tilde B^{2}_{5,4}=a_{2}a_{1}b_{1}^{2}+a_{1}a_{2}b_{1}^{2}+a_{1}^{2}b_{2}b_{1}+a_{1}^{2}b_{1}b_{2}$, $\tilde B_{3,3}^{2}=a_{1}^{2}b_{1}$, $\tilde B_{2,2}^{2}=a_{1}^{2}$, $C_{2,1}=-\frac12 b_{2}$, and
	$C_{3,2}=\frac1{3!}[b_{1},b_{2}]$, we check that $3\tilde B^{2}_{5,4}-3\tilde B^{2}_{3,3}C_{2,1}+3!\tilde B^{2}_{2,2}C_{3,2}=B^{2}_{5,3}$ as expected by Theorem (\ref{th1}).
\end{example}\ \\
In $\NCSF^{(2)}$, we  compute explicitly the polynomial $C_{n,k}$
\begin{myequation}\label{Cnk}
	C_{n,k}= \frac{(-1)^{k}}n\frac{1}{k!(n-k-1)!}
	 \sum_{i_{1},\dots,i_{k}}\binom{n-k-1}{i_{1}-1,\dots,i_{k-1}-1,i_k-2}
	 [b_{i_{1}},[b_{i_{2}},\cdots,[b_{i_{k-1}},b_{i_{k-1}}]\cdots].
\end{myequation}
\begin{example}\rm
	Consider for instance the polynomial $C_{5,2}$ in $\NCSF^{(2)}$
	\begin{myequation}
		\begin{array}{rcl}
			C_{5,2}&=&-\frac1{48} ad_{\partial}^{4} ad_{b_1}^{2} \partial={}
			-\frac1{48} ad_{\partial}^{4} [ b_1,b_{2}]\\
			&=&-\frac1{48}[[[[[b_{1},b_{2}],\partial],\partial],\partial],\partial]\\
			&=&-\frac1{48}(2[b_{3},b_{4}]+3[b_{2},b_{5}]+[b_{1},b_{6}])\\
		&=&-\frac1{48}([b_{5},b_{2}]+4[b_{4},b_{3}]+6[b_{3},b_{4}]+4[b_{2},b_{5}]+[b_{1},b_{6}]).
		\end{array}
	\end{myequation}
\end{example}
\begin{remark}\rm
If we set $a_{i}=b_{i}$ then for each $i$ then we have $ S^{\bullet}(t,x,y) = S^*(y + t,x)$, and so $S(t,x,y)=S^*(y + xt,x) Z(x,t)$.
\end{remark}

In $\Sym^{(2)}$, the series $Z(x,t)$ has a nice closed form
\begin{myequation}\label{cor2}
Z(x,t) =\exp{\left( - \sum_{i \geqslant 2} \frac{(i-1)}{i!} b_{i} t^i \right)}.
\end{myequation}
Indeed,
since the algebra is commutative  $ad_{\partial}^{i} ad_{b_1}^{j} \partial$ is non null only if $j=1$ and when $j=1$ formula (\ref{Cnk}) gives $[\partial,b_{i}]=-b_{i+1}$.

As a consequence, using equality  (\ref{cor2}) together with Theorem \ref{th1} and Formula (\ref{th3Mih}), we find
\begin{myequation}\label{ConsSG}
 S^\bullet(xt,x,y)=\exp\left(\sum_{j \geqslant 0} a_{j+1} \frac{x^j }{j!}y\right)\exp\left( \sum_{j \geqslant 1} jb_j \frac{x^j t}{j!}\right).
\end{myequation}
In other words, equaling the coefficients in the left and the right hand sides of (\ref{ConsSG}) we find \begin{myequation}
\tilde B^{r}_{n+k+r,k+r}={\binom{n+k}{n}}^{-1}B^{r}_{n+k+r,k+r}(a_{1},a_{2},\dots;b_{1},2b_{2},3b_{3},\dots).
\end{myequation}
 In the case where $r=0$, we obtain

\begin{align*}
\tilde B^{0}_{n+k,k}(a_{1},a_{2},\dots;b_{1},b_{2},\dots) &=& B^{k}_{n+k,k}(b_{1},b_{2},\dots;b_{1},b_{2},\dots) \\ 
&=& \binom{n+k}{n}^{-1}B_{n+k,k}(b_{1},2b_{2},3b_{3},\dots).
\end{align*}

\end{document}